\documentclass[11pt]{article}
\usepackage{amsmath,amsthm,amssymb}
\usepackage[dvips]{graphicx}

\newcommand{\N}{\mathbb N}
\newcommand{\Z}{\mathbb Z}
\newcommand{\R}{\mathbb R}
\newcommand{\C}{\mathbb C}
\newcommand{\HH}{{\cal H}}
\newcommand{\p}{{\cal P}}
\newcommand{\PP}{{\mathbb P}}
\newcommand{\T}{{\cal T}}
\newcommand{\U}{{\cal U}}
\newcommand{\tn}{\mathbf{T}_n}

\newcommand{\re}{\operatorname{Re}}

\newcommand{\ii}{{\operatorname{i}}}

\newcommand{\sn}{\operatorname{sn}}
\newcommand{\cn}{\operatorname{cn}}
\newcommand{\dn}{\operatorname{dn}}
\newcommand{\ov}{\overline}
\newcommand{\ta}{\tilde{a}}
\newcommand{\trho}{\tilde{\varrho}}
\newcommand{\iK}{\ii{K}'}

\begin{document}


\title{A Density Result Concerning Inverse Polynomial Images\footnote{will be published in: Proceedings of the AMS (2013).}}
\author{Klaus Schiefermayr\footnote{University of Applied Sciences Upper Austria, School of Engineering and Environmental Sciences, Stelzhamerstr.\,23, 4600 Wels, Austria, \textsc{klaus.schiefermayr@fh-wels.at}}}
\date{}
\maketitle

\theoremstyle{plain}
\newtheorem{theorem}{Theorem}
\newtheorem{lemma}{Lemma}
\theoremstyle{definition}
\newtheorem*{remark}{Remark}

\begin{abstract}
In this paper, we consider polynomials of degree $n$, for which the inverse image of $[-1,1]$ consists of two Jordan arcs. We prove that the four endpoints of these arcs form an ${\cal O}(1/n)$-net in the complex plane.
\end{abstract}



\section{Introduction and Main Result}


Let $\PP_n$ be the set of all polynomials of degree $n$ with complex coefficients. For a polynomial $\T_n\in\PP_n$, consider the inverse image of $[-1,1]$ defined by
\begin{equation}
\T_n^{-1}([-1,1]):=\bigl\{z\in\C:\T_n(z)\in[-1,1]\bigr\}.
\end{equation}
Inverse polynomial images are interesting for instance in approximation theory, since each polynomial $\T_n\in\PP_n$ is (suitable normed) the Chebyshev polynomial, i.e.\ the minimal polynomial with respect to the supremum norm, of degree $n$ on its inverse image $\T_n^{-1}([-1,1])$, see \cite{KaBo-1994}, \cite{OPZ-1996}, \cite{Peh-1996},  or \cite{FiPeh-2001}.

In the following, we will need the notion of Jordan arcs. A set $\bigl\{\gamma(t)\in\ov{\C}:t\in[0,1]\bigr\}$ is called a \emph{Jordan arc} if $\gamma:[0,1]\to\ov{\C}$ is continuous and $\gamma:[0,1)\to\ov{\C}$ is injective.

It is well known that $\T_n^{-1}([-1,1])$ is the union of $n$ Jordan arcs. The number of Jordan arcs can be reduced for some polynomials $\T_n$: the inverse image $\T_n^{-1}([-1,1])$ consists of one Jordan arc if and only if $\T_n(z)=T_n(az+b)$, where $T_n$ is the classical Chebyshev polynomial of the first kind, i.e.\ $T_n(z)=\cos(n\arccos{z})$, see \cite[Cor.\,1]{Sch-2012a}. In this case the inverse image is an interval in the complex plane (which of course can be seen as the union of $n$ intervals). In other words, the case of one Jordan arc is trivial.

In this note, we are interested in polynomials with an inverse image consisting of \emph{two Jordan arcs}. Given four pairwise distinct points $a_1,a_2,a_3,a_4\in\C$ in the complex plane, define
\begin{equation}\label{H}
\HH_4(z):=(z-a_1)(z-a_2)(z-a_3)(z-a_4).
\end{equation}
Then the following characterization theorem holds, see \cite[Thm.\,1]{Sch-2012a} or \cite[Thm.\,3]{PehSch-2004}.


\begin{theorem}\label{Thm-TwoArcs}
Let $\T_n(z)=\tau{z}^n+\ldots\in\PP_n$ be any polynomial of degree $n$. Then $\T_n^{-1}([-1,1])$ consists of two (but not less than two) Jordan arcs with endpoints $a_1,a_2,a_3,a_4$ if and only if $\T_n^2-1$ has exactly $4$ pairwise distinct zeros $a_1,a_2,a_3,a_4$ of odd multiplicity, i.e., if and only if $\T_n$ satisfies a polynomial equation of the form
\begin{equation}\label{TU}
\T_n^2(z)-1=\HH_4(z)\,\U_{n-2}^2(z)
\end{equation}
with $\U_{n-2}(z)=\tau{z}^{n-2}+\ldots\in\PP_{n-2}$ and $\HH_4$ given in \eqref{H}. In this case, the tuple\\ $(a_1,a_2,a_3,a_4)\in\C^4$ is called a $\tn$-tuple.
\end{theorem}


Condition \eqref{TU} implies that $\T_n^{-1}([-1,1])$ consists of two Jordan arcs, which are not necessarily analytic. Concerning the minimum number of \emph{analytic} Jordan arcs, we refer to \cite[Thm.\,3]{Sch-2012a}.

Now, we are able to state the main result. Roughly spoken it says that all $\tn$-tuples $(a_1,a_2,a_3,a_4)$ form an ${\cal O}(\frac{1}{n})$-net in the complex plane.


\begin{theorem}\label{MainThm}
Let $a_1,a_2,a_3,a_4\in\C$ be four paiwise distinct points in the complex plane. Then there exist $\ta_2,\ta_3\in\C$ such that $(a_1,\ta_2,\ta_3,a_4)$ is a $\tn$-tuple and
\[
|a_2-\ta_2|\leq\frac{C_1}{n} \qquad \text{and} \qquad |a_3-\ta_3|\leq\frac{C_2}{n}
\]
holds for $n\geq{N}$, where $C_1,C_2,N$ depend only on $a_1,a_2,a_3,a_4$ but do not depend on $n$.
\end{theorem}


\begin{remark}
\begin{enumerate}
\item The proof of Theorem\,\ref{MainThm} is given in the next section.
\item The values $C_1$ and $C_2$ can be expressed using Jacobian elliptic functions, see \eqref{C1} and \eqref{C2} in the proof of Theorem\,\ref{MainThm}, respectively.
\item For the real case, i.e.\ $a_1,a_2,a_3,a_4\in\R$, the density result of Theorem\,\ref{MainThm} has been proved long time ago by Achieser\,\cite{Achieser-1932}.
\item For the special case $a_1,a_2\in\R$, $a_3\in\C\setminus\R$, $a_4=\overline{a}_3$, the density result of Theorem\,\ref{MainThm} is proved in \cite{Sch-2009}.
\item For the case of $\ell$ real intervals, a similar density result to that of Theorem\,\ref{MainThm} has been proved independently about the same time by Bogatyrev\,\cite{Bogatyrev-1999}, Peherstorfer\,\cite{Peh-2001}, and Totik\,\cite{Totik-2001}.
\end{enumerate}
\end{remark}

\section{Proof of the Main Result}


The proof of the main result is managed with the help of the characterization of a $\tn$-tuple using Jacobian elliptic and theta functions, see \cite{PehSch-2004}. To this end, let us briefly recall some definitions. For an introduction to elliptic functions and integrals, we refer to \cite{BF} and \cite{Lawden}.

Let $k\in{D}_k$, $D_k$ defined in \eqref{Dk}, be the modulus of the Jacobian elliptic functions $\sn(u)=\sn(u,k)$, $\cn(u)=\cn(u,k)$, and $\dn(u)=\dn(u,k)$. Let $k'$, defined by ${k'}^2:=1-k^2$, be the complementary modulus and let $K=K(k)$ the complete elliptic integral of the first kind and let $K'=K'(k):=K(k')$.

Let $a_1,a_2,a_3,a_4\in\C$ be four pairwise distinct complex points and define the modulus $k$ by
\begin{equation}\label{k}
k^2:=\frac{(a_4-a_1)(a_3-a_2)}{(a_4-a_2)(a_3-a_1)}.
\end{equation}
By optionally exchanging some of the $a_i$'s, it is always possible to get a tuple $(a_1,a_2,a_3,a_4)$ which satisfies
\begin{equation}\label{ineq-a1a2a3a4}
|a_4-a_1|\cdot|a_3-a_2|\leq|a_4-a_2|\cdot|a_3-a_1|,
\end{equation}
i.e., $|k^2|\leq1$, and for which $k^2\notin\R^-$. Obviously $k^2=0$ if and only if $(a_1=a_4 \vee a_2=a_3)$ and $k^2=1$ if and only if $(a_1=a_2\vee{a}_3=a_4)$. In order to get $k$, we choose that branch of the square root in \eqref{k}, for which $\re(k)>0$. Therefore, we have to consider the case
\begin{equation}\label{Dk}
k\in{D}_k:=\bigl\{u\in\C:|u|\leq1,\re(u)>0,u\neq1\bigr\}
\end{equation}
in the following. For $k\in{D}_k$, the functions $K=K(k)$ and $K'=K'(k)$ are single valued (for a detailed discussion see \cite[Sect.\,8.12]{Lawden}). For the complementary modulus $k'$, by \eqref{k} and relation ${k'}^2=1-k^2$, we get
\begin{equation}\label{k'}
{k'}^2=\frac{(a_4-a_3)(a_2-a_1)}{(a_4-a_2)(a_3-a_1)}.
\end{equation}
Further, we will need the function $\sn^2(u)$, which is an even elliptic function of order $2$ with fundamental periods $2K$ and $2\ii K'$, with a double zero at $0$ and a double pole at $\ii{K}'$. Further, for $k\in{D}_k$, let
\begin{equation}\label{P}
{\cal{P}}:=\bigl\{\mu{K}+\ii\mu'K':0\leq\mu,\mu'\leq1\vee(0<\mu<1\wedge-1<\mu'<0)\bigr\}
\end{equation}
be a ``half'' period parallelogram of $\sn^2(u)$ with respect to the modulus $k$. Note that $\iK/K\notin\R$. By the above mentioned properties of $\sn^2(u)$, the mapping $\sn^2:{\cal{P}}\to\ov\C$, $u\mapsto\sn^2(u)$, is bijective, hence for given $a_1,a_2,a_3,a_4\in\C$, there exists a unique $\varrho=\lambda{K}+\ii\lambda'K'\in\p$ with $\lambda,\lambda'\in\R$ such that
\begin{equation}\label{sn}
\sn^2(\varrho)=\sn^2(\lambda{K}+\ii\lambda'K')=\frac{a_4-a_2}{a_4-a_1}.
\end{equation}
Note that $\varrho=0,K,K+\iK,\iK$ is equivalent to $\sn^2(\varrho)=0,1,1/k^2,\infty$, respectively, and that $(a_4-a_2)/(a_4-a_1)=0,1,1/k^2,\infty$ is posssible only if $a_4=a_2$, $a_2=a_1$, $a_2=a_1$, $a_4=a_1$, respectively. Thus, since $a_1,a_2,a_3,a_4$ are pairwise distinct, we have $\varrho\notin\{0,K,K+\iK,\iK\}$.

In the following characterization theorem\,\cite[Thm.\,7]{PehSch-2004}, a necessary and sufficient condition is given such that $(a_1,a_2,a_3,a_4)$ is a $\tn$-tuple.


\begin{theorem}\label{ChThm}
Let $n\in\N$, let $a_1,a_2,a_3,a_4\in\C$ be pairwise distinct and satisfy \eqref{ineq-a1a2a3a4}, and let $k\in{D}_k$ and $\varrho\in\p$ be defined by \eqref{k} and \eqref{P}, respectively. Then $(a_1,a_2,a_3,a_4)$ is a $\tn$-tuple if and only if $\varrho$ is of the form
\begin{equation}\label{Rho}
\varrho=\tfrac{m}{n}K+\ii\,\tfrac{m'}{n}K',\quad\text{where}\quad{m},m'\in\Z.
\end{equation}
\end{theorem}


In Fig.\,\ref{Fig_P}, the set $\p$, defined in \eqref{P}, and all points $\varrho\in\p$ of the form \eqref{Rho} (where $n=6$ was chosen) are illustrated.


\begin{figure}[ht]
\begin{center}
\includegraphics[scale=0.7]{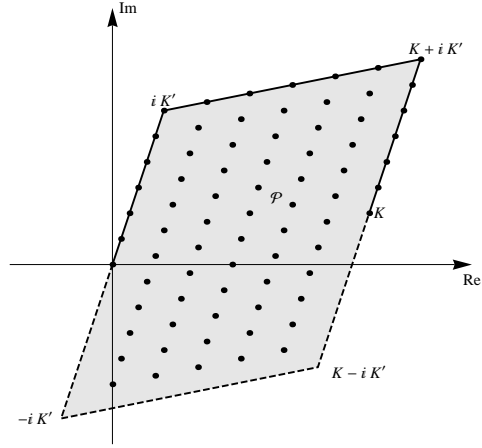}
\caption{\label{Fig_P} Illustration of the parallelogram $\p$}
\end{center}
\end{figure}

If \eqref{Rho} holds, i.e.\ if $(a_1,a_2,a_3,a_4)$ is a $\tn$-tuple, then the corresponding polynomials $\T_n(z)$, $\U_{n-2}(z)$, and $\HH_4(z)$ of Theorem\,\ref{Thm-TwoArcs} can be represented with a certain conformal mapping and with the help of Jacobi's elliptic and theta functions, for details we refer to \cite{PehSch-2004}.

Before we start with the proof of Theorem\,\ref{MainThm}, let us state an inequality for the elliptic sine function, which will be crucial for the proof.


\begin{lemma}\label{lem}
Let $k\in{D}_k$, and $u\in\C$, $|u|\leq\tfrac{\pi}{4}$. Then
\[
|\sn(u)|\leq\tan|u|\leq\tfrac{4}{\pi}|u|.
\]
\end{lemma}
\begin{proof}
A more general version of the first inequality is proved in \cite{Sch-2012b}. The second inequality follows immediately from $f(x):=\tan{x}-\frac{4}{\pi}x\leq0$, $0\leq{x}\leq\frac{\pi}{4}$, since $f(0)=f(\frac{\pi}{4})=0$ and $f''(x)=2\sin{x}/\cos^3x>0$ for $0\leq{x}\leq\frac{\pi}{4}$.
\end{proof}


\begin{proof}[{\bf Proof of Theorem\,\ref{MainThm}}]
Let $n\in\N$ and let us assume that $a_1,a_2,a_3,a_4$ satisfy inequality \eqref{ineq-a1a2a3a4} (which is always possible by a reordering of the $a_j$). Let $k$ be defined by \eqref{k}, let $\varrho\in{\cal{P}}$ be defined by equation \eqref{sn} and let $\lambda,\lambda'\in\R$ be uniquely defined by
\begin{equation}\label{rho}
\varrho=\lambda{K}+\ii\lambda'K'.
\end{equation}
By \eqref{k}, \eqref{k'}, and \eqref{sn},
\begin{equation}\label{a2}
a_2=a_4-(a_4-a_1)\sn^2(\varrho)
\end{equation}
and
\begin{equation}\label{a3}
a_3=\frac{a_4-{k'}^2a_1}{k^2}\cdot\frac{a_2+\dfrac{k^2a_1a_4}{{k'}^2a_1-a_4}}{a_2+\dfrac{{k'}^2a_4-a_1}{k^2}}=A_1\cdot\frac{a_2+A_2}{a_2+A_3},
\end{equation}
where
\begin{equation}\label{A1A2A3}
A_1:=\frac{a_4-{k'}^2a_1}{k^2}, \quad A_2:=\frac{k^2a_1a_4}{{k'}^2a_1-a_4}, \quad A_3:=\dfrac{{k'}^2a_4-a_1}{k^2}.
\end{equation}
Clearly there exist integers $m,m'\in\Z$ such that
\begin{equation}\label{mn}
\Bigl|\frac{m}{n}-\lambda\Bigr|\leq\frac{1}{n}\qquad\text{and}\qquad\Bigl|\frac{m'}{n}-\lambda'\Bigr|\leq\frac{1}{n}.
\end{equation}
Let us remark that the integers $m,m'$ can be chosen such that (cf.\ Fig.\,\ref{Fig_P})
\begin{equation}\label{rhopmtrho}
\varrho\pm\trho\notin\{\nu{K}+\ii\,\nu'K':\nu,\nu'\in\Z\},
\end{equation}
where $\trho$ is defined by
\begin{equation}\label{rho-t}
\trho:=\tfrac{m}{n}\,K+\ii\,\tfrac{m'}{n}K'.
\end{equation}
Note that \eqref{rhopmtrho} implies that none of the points $\varrho\pm\trho$ or $\tfrac{1}{2}(\varrho\pm\trho)$ is a pole of $\sn(u)$, $\cn(u)$, or $\dn(u)$. Define
\begin{equation}\label{a2-t}
\ta_2:=a_4-(a_4-a_1)\sn^2(\trho)
\end{equation}
and
\begin{equation}\label{a3-t}
\ta_3:=\frac{a_4-{k'}^2a_1}{k^2}\cdot\frac{\ta_2+\dfrac{k^2a_1a_4}{{k'}^2a_1-a_4}}{\ta_2+\dfrac{{k'}^2a_4-a_1}{k^2}}
=A_1\cdot\frac{\ta_2+A_2}{\ta_2+A_3},
\end{equation}
where the last equality follows from \eqref{A1A2A3}. Then
\begin{equation}\label{k-t}
k^2=\frac{(a_4-a_1)(\ta_3-\ta_2)}{(a_4-\ta_2)(\ta_3-a_1)}
\end{equation}
and
\begin{equation}\label{sn-t}
\sn^2(\trho)=\frac{a_4-\ta_2}{a_4-a_1}.
\end{equation}
holds. Note that the modulus $k$ for the tuples $(a_1,a_2,a_3,a_4)$ and $(a_1,\ta_2,\ta_3,a_4)$ is the same. By \eqref{rho-t}, \eqref{k-t}, \eqref{sn-t} and Theorem\,\ref{ChThm}, $(a_1,\ta_2,\ta_3,a_4)$ is a $\tn$-tuple. By \eqref{a2} and \eqref{a2-t},
\begin{align*}
&|a_2-\ta_2|=|a_4-a_1|\cdot|\sn^2(\varrho)-\sn^2(\trho)|\\
&\quad=|a_4-a_1|\cdot|\sn(\varrho)-\sn(\trho)|\cdot|\sn(\varrho)+\sn(\trho)|\\
&\quad=|a_4-a_1|\cdot|\sn(\varrho)+\sn(\trho)|\cdot
\frac{|2\sn(\frac{1}{2}\varrho-\frac{1}{2}\trho)\cn(\frac{1}{2}\varrho+\frac{1}{2}\trho)\dn(\frac{1}{2}\varrho+\frac{1}{2}\trho)|}
{|1-k^2\sn^2(\frac{1}{2}\varrho+\frac{1}{2}\trho)\sn^2(\frac{1}{2}\varrho-\frac{1}{2}\trho)|},
\end{align*}
where in the last equation the well-known formula for $\sn(u)-\sn(v)$ is used, see, e.g., \cite[(123.06)]{BF}. By \eqref{rho}, \eqref{mn}, and \eqref{rho-t},
\begin{align}
|\varrho-\trho|&=|(\lambda-\tfrac{m}{n})K+\ii(\lambda'-\tfrac{m'}{n})K'|\notag\\
&\leq|\lambda-\tfrac{m}{n}|\cdot|K|+|\lambda'-\tfrac{m'}{n}|\cdot|K'|\notag\\
&\leq\frac{|K|+|K'|}{n}.\label{KKn}
\end{align}
If
\begin{equation}
n\geq\tfrac{2}{\pi}(|K|+|K'|)=:n_1
\end{equation}
then, by \eqref{KKn},
\begin{equation}\label{rho-trho}
|\tfrac{1}{2}\varrho-\tfrac{1}{2}\trho|\leq\frac{\pi}{4}.
\end{equation}
Thus, using Lemma\,\ref{lem} and \eqref{KKn},
\[
|\sn(\tfrac{1}{2}\varrho-\tfrac{1}{2}\trho)|\leq\tfrac{4}{\pi}|\tfrac{1}{2}\varrho-\tfrac{1}{2}\trho|
\leq\tfrac{2}{\pi}(|K|+|K'|)\tfrac{1}{n}.
\]
Summing up, for $n\geq{n}_1$, we have the inequality
\[
|a_2-\ta_2|\leq\frac{2}{n\pi}(|K|+|K'|)|a_4-a_1|B.
\]
where
\begin{equation}\label{B}
B:=|\sn(\varrho)+\sn(\trho)|\cdot
\frac{|2\cn(\frac{1}{2}\varrho+\frac{1}{2}\trho)\dn(\frac{1}{2}\varrho+\frac{1}{2}\trho)|}
{|1-k^2\sn^2(\frac{1}{2}\varrho+\frac{1}{2}\trho)\sn^2(\frac{1}{2}\varrho-\frac{1}{2}\trho)|}
\end{equation}
Since $\varrho\notin\{0,K,\pm\iK,K\pm\iK\}$, there exists an $n_2\in\N$ such that $0,K,\pm\iK,K\pm\iK\notin\p(n_2)$, where
\[
\p(n_2):=\left\{u\in\C:u=\mu{K}+\ii\mu'K',\,|\mu-\lambda|\leq\frac{1}{n_2},\,|\mu'-\lambda'|\leq\frac{1}{n_2},\,\mu,\mu'\in\R\right\}.
\]
Thus, the maxima
\[
s^*:=\max_{u\in\p(n_2)}|\sn(u)|, \qquad c^*:=\max_{u\in\p(n_2)}|\cn(u)|, \qquad d^*:=\max_{u\in\p(n_2)}|\dn(u)|
\]
exist. By construction of $\p(n_2)$, for $n\geq{n}_2$, we have $\varrho,\trho,\tfrac{1}{2}(\varrho+\trho)\in\p(n_2)$. Further, we have
\[
\left|k^2\right|\cdot\left|\sn^2(\tfrac{1}{2}\varrho+\tfrac{1}{2}\trho)\right|\cdot\left|\sn^2(\tfrac{1}{2}\varrho-\tfrac{1}{2}\trho)\right|
\leq(s^*)^2\cdot\frac{4}{\pi^2}(|K|+|K'|)^2\cdot\frac{1}{n^2}\leq\frac{1}{2},
\]
where the last inequality is true if
\[
n\geq n_3:=\frac{2\sqrt{2}\,s^*}{\pi}(|K|+|K'|).
\]
Hence, for $B$ defined in \eqref{B},
\[
B\leq\frac{4s^*c^*d^*}
{\left|1-\left|k^2\right|\cdot\left|\sn^2(\tfrac{1}{2}\varrho+\tfrac{1}{2}\trho)\right|
\cdot\left|\sn^2(\tfrac{1}{2}\varrho-\tfrac{1}{2}\trho)\right|\right|}
\leq8s^*c^*d^*
\]
and altogether, for $n\geq\max\{n_1,n_2,n_3\}$, we get the inequality
\[
|a_2-\ta_2|\leq\frac{C_1}{n},
\]
where
\begin{equation}\label{C1}
C_1:=\frac{16}{\pi}\left(|K|+|K'|\right)|a_4-a_1|\,s^*c^*d^*.
\end{equation}
Using \eqref{a3} and \eqref{a3-t}, we get
\[
|a_3-\ta_3|=|A_1|\cdot\frac{|a_2-\ta_2|\cdot|A_2-A_3|}{|a_2+A_3|\cdot|\ta_2+A_3|}
\]
and
\[
|\ta_2+A_3|=|a_4-a_1|\cdot|\sn^2(\trho)-1/k^2|.
\]
Since $u=K+\iK\in\p$ is the only point in $\p$, for which $\sn^2(u)=1/k^2$, by construction of $\p(n_2)$,
\[
s^{**}:=\min_{u\in\p(n_2)}|\sn^2(u)-1/k^2|>0
\]
holds. Thus
\[
|a_3-\ta_3|\leq\frac{C_2}{n},
\]
where
\begin{equation}\label{C2}
C_2:=\frac{C_1\cdot|A_1|\cdot|A_2-A_3|}{|a_2+A_3|\cdot|a_4-a_1|\cdot{s}^{**}}
\end{equation}
and $A_1,A_2,A_3$ are defined in \eqref{A1A2A3}.
\end{proof}


\bibliographystyle{amsplain}
\bibliography{DensityTwoArcs}

\end{document}